\documentclass[12pt]{amsart}
\usepackage{amsmath}
\usepackage{amscd}
\usepackage{amssymb}
\usepackage{amsfonts}
\usepackage{amsthm}
\usepackage{bbm}
\usepackage{cancel}
\usepackage{color}
\usepackage{eucal}
\usepackage{enumerate,yfonts}
 
 \usepackage{fullpage}
\usepackage{pdfsync}
 \usepackage[all,cmtip]{xy}
\usepackage{graphicx}
\usepackage{graphics}
\usepackage{hyperref}
\usepackage{latexsym}
\usepackage{mathrsfs}
 \usepackage{placeins}
\usepackage{pstricks}
\usepackage{bookmark}
\usepackage{mathtools}

\usepackage{stmaryrd}
\usepackage{url}

\newtheorem{thm}{Theorem}[section]
\newtheorem{theorem}[thm]{Theorem}
\newtheorem{corollary}[thm]{Corollary}
\newtheorem{lemma}[thm]{Lemma}

\newtheorem{proposition}[thm]{Proposition}
\newtheorem{prop}[thm]{Proposition}

\newtheorem{thm-dfn}[thm]{Theorem-Definition}
\newtheorem{cor}[thm]{Corollary}

\theoremstyle{definition}

\newtheorem{defn}{Definition}[section]

\theoremstyle{remark}
\newtheorem{remark}{Remark}[section]
\newtheorem{example}[remark]{Example}

\newtheorem{rmk}[remark]{Remark}

\numberwithin{equation}{section}

\newcommand{\qbin}[3]{\genfrac{[}{]}{0pt}{}{#1}{#2}_{#3}}

\newcommand{\fg}{{\mathfrak g}}

\newcommand{\fa}{{\mathfrak a}}

\newcommand{\fF}{{\mathfrak{F}}}

\newcommand{\Lg}{{\mathfrak g}}

\newcommand{\bC}{{\mathbb C}}

\newcommand{\bZ}{{\mathbb Z}}

\newcommand{\bP}{{\mathbb P}}

\newcommand{\mF}{\mathcal{F}}

\newcommand{\mO}{\mathcal{O}}

\newcommand{\mH}{\mathcal{H}}

\newcommand{\calF}{{\mathcal F}}

\newcommand{\calL}{{\mathcal L}}

\newcommand{\calW}{{\mathcal W}}

\newcommand{\cW}{{\mathcal W}}

\newcommand{\cO}{{\mathcal O}}

\newcommand{\cF}{{\mathcal F}}
\newcommand{\cN}{{\mathcal N}}

\newcommand{\on}{\operatorname}

\newcommand{\ra}{\rightarrow}

\newcommand{\Loc}{\on{LocSys}}

\newcommand{\nc}{\newcommand}

\nc{\al}{{\alpha}} \nc{\be}{{\beta}} \nc{\ga}{{\gamma}}
\nc{\ve}{{\varepsilon}} \nc{\Ga}{{\Gamma}} 
\nc{\La}{{\Lambda}}

\nc{\ad }{{\on{ad }}}

\nc{\aff}{{\on{aff}}} \nc{\Aff}{{\mathbf{Aff}}}

\nc{\der}{{\on{der}}}

\nc{\diag}{{\on{diag}}}

\nc{\Fl}{{\calF\ell}}

\nc{\Hg}{{\on{Higgs}}}

\nc{\Id}{{\on{Id}}}

\nc{\Ind}{{\on{Ind}}}

\nc{\Op}{{\on{Op}}}

\nc{\res}{{\on{res}}}

\nc{\tr}{{\on{tr}}}

\nc{\GSp}{{\on{GSp}}} \nc{\GU}{{\on{GU}}} \nc{\SL}{{\on{SL}}}
\nc{\SU}{{\on{SU}}} \nc{\SO}{{\on{SO}}}

\nc{\nh}{{\Loc_{J^p}(\tau')}}
\nc{\bnh}{{\Loc_{\breve J^p}(\tau')}}

\nc{\bU}{{\overline{U}}} 
\nc{\IC}{{\on{IC}}}

\newcommand{\Jac}{\mathrm{Jac}}

\newcommand{\rSO}{\mathrm{SO}}
\newcommand{\rSL}{\mathrm{SL}}

\newcommand{\p}{\perp}

\nc{\ot}{\otimes}

\nc{\oh}{{\operatorname{H}}}
\nc{\gr}{{\operatorname{gr}}}
\nc{\rk}{{\operatorname{rank}}}
\nc{\codim}{{\operatorname{codim}}}
\nc{\img}{{\operatorname{Im}}}
\nc{\Span}{{\operatorname{Span}}}
\nc{\Img}{\operatorname{Im}}

\newcommand{\beqn}{\begin{equation*}}
\newcommand{\eeqn}{\end{equation*}}

\newcommand{\beq}{\begin{equation}}
\newcommand{\eeq}{\end{equation}}

\newcommand{\bern}{\begin{eqnarray*}}
\newcommand{\eern}{\end{eqnarray*}}
\nc{\Fano}[1]{\on{Fano}_{#1}}

\begin{document}
\title{on the cohomology of Fano varieties and the Springer correspondence}
        \author{Tsao-Hsien Chen}
        \address{Department of Mathematics, University of Chicago, Chicago, IL 60637, USA.}
        \email{chenth@math.uchicago.edu}
        \thanks{Tsao-Hsien Chen was supported in part by the AMS-Simons travel grant.}
        \author{Kari Vilonen}
        \address{School of Mathematics and Statistics, University of Melbourne, VIC 3010, Australia, and Department of Mathematics and Statistics, University of Helsinki, Helsinki, 00014, Finland}
         \email{kari.vilonen@unimelb.edu.au, kari.vilonen@helsinki.fi}
         \thanks{Kari Vilonen was supported in part by NSF grant DMS-1402928, the Academy of Finland,  the ARC grant DP150103525, the Humboldt Foundation, and the Simons Foundation.}
         \author{Ting Xue}
         \address{ School of Mathematics and Statistics, University of Melbourne, VIC 3010, Australia, and Department of Mathematics and Statistics, University of Helsinki, Helsinki, 00014, Finland}
         \email{ting.xue@unimelb.edu.au}
\thanks{Ting Xue was supported in part by the ARC grants DP150103525, DE160100975 and the Academy of Finland.}

\makeatletter
\let\@wraptoccontribs\wraptoccontribs
\makeatother
\contrib[with an appendix by]{Dennis Stanton}

\begin{abstract}
In this paper we compute the cohomology of the Fano varieties of $k$-planes in the smooth complete intersection of two quadrics in $\bP^{2g+1}$, using Springer theory for symmetric spaces.
\end{abstract}

\maketitle

\section{Introduction}

In this paper we compute the cohomology of the Fano varieties $\Fano k$ of $k$-planes  in the smooth complete intersection of two quadrics in $\bP^{2g+1}$, with $g\geq 1$. These Fano varieties have a concrete interpretation as moduli spaces of vector bundles (with extra structure) on a hyperelliptic curve $C$ of genus $g$. When $k=g-1$ then $\Fano {g-1}= \on{Jac}(C)$, the Jacobian of $C$ \cite{Re,Do} and when $k=g-2$ then $\Fano {g-2}= \on{Bun}_2(C)$, the moduli space of stable rank 2 vector bundles on $C$ with fixed odd determinant~\cite{DR}. For $k < g-2$ a more elaborated interpretation of the varieties $\Fano k$ as moduli spaces of bundles is given in~\cite{Ra}. The curve $C$ arises from the intersection of two quadrics in the following manner. If the intersection of the two quadrics is given by the pencil $\mu Q_1+\lambda Q_2$ then $C$ is the hyperelliptic curve over $\bP^1$ ramified at the points $[\mu,\lambda]$ where the quadric $\mu Q_1+\lambda Q_2$ becomes singular. 

The goal of this paper is to describe the cohomology of the varieties $\Fano k$ in general.  The form of our answer is in the spirit of the main theorem of~\cite{N} who, from our point of view, treats the case $\Fano {g-2}$. He makes use of the mapping class group which for us, as we work with hyperelliptic curves, is replaced by the fundamental group of the universal family of hyperelliptic curves. 

To state our result note that $\dim\Fano {g-i}=(g-i+1)(2i-1)$. We also  write
\beqn
\bar{H}^k(\Fano {g-i},\bC)=H^{\dim\Fano {g-i}-k}(\Fano {g-i},\bC),\ \ \bar\wedge^{k}\left(H^1(C,\bC)\right)=\wedge^{g-k}\left(H^1(C,\bC)\right).
\eeqn

\begin{thm}\label{thm-fano}
For $i\geq 2$, we have
\beqn
\bar{H}^k(\Fano {g-i},\bC)\cong\bigoplus_{j=i-1}^{g} N_i(k,j)\,\bar\wedge^{j}\left(H^1(C,\bC)\right),
\eeqn
where $N_i(k,j)$ is the coefficient of $q^k$ in 
\beqn
q^{-(j-i+1)(2i-1)}(1-q^{4j})\frac{\prod_{l=j-i+2}^{i+j-2}(1-q^{2l})}{\prod_{l=1}^{2i-2}(1-q^{2l})}.
\eeqn
In particular, the numbers $N_i(k,j)$ are independent of the genus $g$.
\end{thm}
To come up with this formula we were inspired by the case of $\Fano {g-2}=\on{ Bun}_2(C)$ treated by Nelson \cite{N}. He, in turn, states that the formula in the case of $\Fano {g-2}= \on{Bun}_2(C)$  was conjectured by Donaldson. We guessed  the general formula above after reading \cite{Z} augmented by some experimentation. By our methods we reduced the proof of the theorem to a combinatorial identity which expresses the Poincare polynomial of the orthogonal Grassmannian in terms of the Poincare polynomials of ordinary Grassmannians in a particular manner, see formula~\eqref{inductive eqn}. This combinatorial identity was proved by Stanton. His proof is included in this paper as an appendix. We have not been able to understand this combinatorial identity from a geometric point of view. It would be interesting to have such a geometric interpretation. 

As a byproduct one obtains formulas for the Poincare polynomials (denoted by $h_j^{(n)}(q)$ in the text) of stalks of IC sheaves of certain nilpotent orbits in the symmetric space case. In the classical case these stalks are given by Kostka--Foulkes polynomials. Thus, the polynomials $h_j^{(n)}(q)$ can be regarded as symmetric space analogues of Kostka--Foulkes polynomials. The appendix of Stanton gives an explicit formula for the $h_j^{(n)}(q)$ and also proves some interesting recursive formulas for them.

The proof of the theorem makes use of Springer theory for symmetric spaces which was initiated in  \cite{CVX}. In particular, it can be seen as a concrete example of a general strategy to compute the cohomology of Hessenberg varieties \cite{GKM} ($\Fano {i}$ is a special case of a Hessenberg variety) using the Fourier transform. 
In \cite{CVX} we primarily worked with the symmetric pair $(\mathrm{SL}(N),\mathrm{SO}(N))$ in the case when $N$ is odd. As an application of our theory we computed the cohomology of Fano varieties of $k$-planes in the complete intersection of two quadrics in an even dimensional projective space~\cite[Theorem 8.1]{CVX}. In the even dimensional case the answer is considerably simpler than in the case of odd dimensional projective space we treat here. In the even dimensional case the cohomology turns out to be Hodge-Tate. 

We briefly explain the idea of the proof of the theorem: 

\noindent Step 1: 
Instead of computing the cohomology $H^k(\Fano {i},\bC)$ of a single Fano variety $\Fano {i}$,
we put $\Fano {i}$ into a proper smooth family 
$\pi:\widetilde{\on{Fano}}_{i}\ra\fg_1^{rs}$ and 
study the decomposition of the local system $R^k\pi_{*}\bC$ into irreducibles.
Here $\fg_1$ is the vector space 
of symmetric matrices in $\mathfrak{sl}({2n})$ ($n=g+1$)
and $\fg_1^{rs}$ is the open subset consisting of regular semisimple elements.

\noindent Step 2: We consider the intersection cohomology complex $\IC(\fg_1,R^k\pi_{*}\bC)$ associated to 
the local system $R^k\pi_{*}\bC$. Since $R^k\pi_{*}\bC$ is semisimple, the irreducible summands of $\IC(\fg_1,R^k\pi_{*}\bC)$ 
are in bijection with those of $R^k\pi_{*}\bC$. Hence it suffices to study the decomposition of 
\linebreak $\IC(\fg_1,R^k\pi_{*}\bC)$ into irreducibles.
The advantage of working with $\IC(\fg_1,R^k\pi_{*}\bC)$ 
is that we can make use of the powerful Fourier transform 
$\mathfrak{F}:\on{Perv}(\fg_1,\bC)\cong\on{Perv}(\fg_1,\bC)$. We show that 
$\mathfrak{F}\left(\IC(\fg_1,R^k\pi_{*}\bC)\right)$ is supported on the nilpotent cone $\mathcal N_1$
in $\fg_1$ and has the form 
\beq\label{iso}
\mathfrak{F}\left(\IC(\fg_1,R^k\pi_{*}\bC)\right)\cong\bigoplus_{j=0}^{i+1} \IC\left(\bar\mO_{2^j1^{2n-2j}},\bC^{\oplus s_{jk}}\right)
\eeq
where $s_{jk}\in\bZ_{\geq 0}$ and 
$\mO_{2^j1^{2n-2j}}$ are nilpotent orbits of order two in $\mathcal N_1$ (see \S\ref{order 2}).

\noindent Step 3: We prove  
an explicit formula for those $s_{jk}$ (see \S\ref{numbers}) and compute the Fourier transform of 
$\IC(\bar\mO_{2^i1^{2n-2i}},\bC)$. We show that the Fourier transform is given by
\[\mathfrak F\left(\IC(\bar\mO_{2^i1^{2n-2i}},\bC)\right)\cong\IC(\fg_1,\cW_i)\] 
where $\cW_i$ is an irreducible local system on $\fg_1^{rs}$ coming from the 
universal family of 
hyperelliptic curves of genus $n-1$ over $\fg_1^{rs}$ 
(see Proposition \ref{prop-matching}). 
This can be viewed as 
an example of Springer correspondence for the symmetric pair $(\rSL(2n),\rSO(2n))$. Now  taking the 
Fourier transform on both sides of  (\ref{iso}) we obtain the decomposition 
of $\IC(\fg_1,R^k\pi_{*}\bC)$ into irreducibles, hence that of $R^k\pi_{*}\bC$. The theorem follows.

It follows from the proof that the isomorphism in Theorem \ref{thm-fano} is compatible with 
the natural actions of the fundamental group of the universal family of hyperelliptic curves 
(or rather the braid group $B_{2n}=\pi_1\left(\fg_1^{rs}\sslash\rSO(2n)\right)$
on both sides.

As we mentioned earlier, our method of computing cohomology of Fano varieties can be 
applied to a more general situation, namely, the case of Hessenberg varieties. 
Those varieties, first introduced in \cite{DPS}, play an important role in 
algebraic geometry and representation theory.
For example, in \cite{CVX1} we show that 
complete intersections of quadrics 
are examples of Hessenberg varieties, and 
in \cite{GKM,OY,T} the authors 
show that many important questions in 
harmonic analysis on $p$-adic groups and  representations of Cherednik algebras are 
closely related to cohomology of Hessenberg varieties. 
It would be very interesting to understand the cohomology of 
those Hessenberg varieties
from the perspective of our method.

The paper is organized as follows. In Section~\ref{first section} we explain how the theory developed in \cite{CVX} for odd $N$ can be carried over to the case of even $N$
and we establish the Springer correspondence 
for nilpotent orbit of order two
for the symmetric pair $(\rSL(2n),\rSO(2n))$.  
In Section~\ref{second section} we give the proof of our main theorem. The appendix by Dennis Stanton contains the proof of the combinatorial identity.

{\bf Acknowledgements.}
We thank the Max Planck Institute for Mathematics in Bonn for support, hospitality, and a nice research environment. Furthermore KV and TX thank the Research Institute for Mathematical Sciences in Kyoto  for support, hospitality, and a nice research environment. Special thanks are due to Dennis Stanton for proving a key combinatorial identity and for writing an appendix containing the proof. We also thank the referee for the careful reading of the paper.

\section{Springer correspondence for the symmetric pair \texorpdfstring{$(\rSL(2n),\rSO(2n))$}{lg}}
\label{first section}

In this section we discuss the Springer correspondence for nilpotent orbits of order 2 for the symmetric pair $(\rSL(2n),\rSO(2n))$. We have treated the case $(\rSL(2n+1),\rSO(2n+1))$ in \cite{CVX}. Some of the arguments in this section are modifications of the arguments in \cite{CVX}. In this section we have written $n$ for $g+1$ and $n \geq 2$.

We follow the notational conventions of \cite{CVX}. In particular, we adopt the usual convention of cohomological degrees for perverse sheaves by having them be symmetric around 0. We also use the convention that all functors are derived, so we write, for example, $\pi_*$ instead of $R\pi_*$. If $X$ is smooth we write $\bC_X[-]$ for the constant sheaf placed in degree $-\dim X$ so that $\bC_X[-]$ is perverse. If $U\subset X$ is a smooth open dense subset of a variety $X$ and $\calL$ is a local system on $U$, we write $\IC(X,\calL)$ for the IC-extension of $\calL[-]$ to $X$; in particular, it is perverse.  For $\mF\in \on{D}(X)$ and $x\in X$, we write $\mH^i_x(\cF)$ for the stalk of the cohomology sheaf $\mH^i\cF$ at $x$. This should not be confused with local cohomology. For a number $a$, we denote the integer part of $a$ by $[a]$.

\subsection{The symmetric pair \texorpdfstring{$(\rSL(2n),\rSO(2n))$}{lg}}\label{order 2}
Let $G=\rSL(2n,\bC)$ and $\theta:G\to G$ an involution such that $K\coloneqq G^\theta=\rSO(2n,\bC)$.  We also write $(G,K) = (\rSL(V),\rSO(V,Q))$, where $\dim\,V = 2n$.  We think of $Q$ concretely as a non-degenerate quadratic form on $V$ and we write $\langle\ ,\rangle_Q$ for the non-degenerate bilinear form on $V$ associated to $Q$. The involution $\theta$ induces a grading on $\Lg=\on{Lie}\,G$, i.e. $\Lg=\Lg_0\oplus\Lg_1$, where $\theta|_{\Lg_i}=(-1)^i$. If we diagonalize $Q$ then the Cartan involution $\theta$ is given by $g\mapsto (g^t)^{-1}$ and then $\Lg_1$ consists of symmetric matrices. 

The pair $(G,K)$ is a split pair.  We write $\Lg^{rs}$ for the regular semi-simple elements in $\Lg$ and we let $\Lg_1^{rs}=\Lg_1\cap\Lg^{rs}$. Furthermore, we write $A$ for a $\theta$-stable  maximal split torus of $G$, i.e., $\theta(t)=t^{-1}$ for $t\in A$. We write $\pi_1^K(\Lg_1^{rs})$ for the equivariant fundamental group of $\Lg_1^{rs}$ and we have
\beqn
\pi_1^K(\Lg_1^{rs})=A[2]\rtimes B_{2n},
\eeqn
where $A[2]$ is the group of order 2 elements in $A$ and $B_{2n}$ is the braid group. For a discussion of these matters see \cite[\S 2.6]{CVX}.

We also write $\cN$ for the nilpotent cone of $\Lg$ and we let $\cN_1=\cN\cap\Lg_1$ stand for  the nilpotent cone of $\Lg_1$. The $G$-orbits in $\cN$ are parametrized by partitions $\lambda$ of $2n$ and we write $\tilde{\cO}_\lambda$ for the $G$-orbit corresponding to the partition $\lambda$. The intersection of a $G$-orbit $\tilde{\cO}_\lambda$ with $\cN_1$ is one $K$-orbit in the case when not all parts of $\lambda$ are even and decomposes into two $K$-orbits otherwise \cite{S}.

For $i\in[0,n-1]$, let $\cO_{2^i1^{2n-2i}}$ 
denote the (unique) nilpotent $K$-orbit in $\Lg_1$ corresponding to the partition $\lambda = 2^i1^{2n-2i}$, where $i$ (resp. $2n-2i$) is the multiplicity of $2$ (resp. $1$) in $\lambda$. Let us denote by $\on{OGr}(s,2n)$ the variety of $s$-dimensional isotropic subspaces in $\bC^{2n}$ with respect to a non-degenerate bilinear form. We write $V_i$ for a vector subspace of $V=\bC^{2n}$ of dimension $i$ and $V_i^\p=\{x\in V\,|\,\langle x,V_i\rangle_Q=0\}$.

\subsection{Reeder's resolutions for \texorpdfstring{$\bar\cO_{2^i1^{2n-2i}}$}{lg}}

Consider the natural projection maps
\beq\label{resolutions}
\upsilon_i:\{(x,0\subset V_i\subset V_i^\p\subset\bC^{2n})\,|\,x\in\Lg_1,\ xV_i^\p=0\}\to\bar\cO_{2^i1^{2n-2i}}.
\eeq
The $\upsilon_i$'s are Reeder's resolutions for $\bar\cO_{2^i1^{2n-2i}}$ \cite{R}. Note that $\cO_{2^j1^{2n-2j}}\subset \bar\cO_{2^i1^{2n-2i}}$ if and only if $j\leq i$.

\begin{lemma}
\label{eqn-fiber}
Let $x_j\in\cO_{2^j1^{2n-2j}}$, $j\leq i$. We have 
\beqn
\upsilon_i^{-1}(x_j)\cong\on{OGr}(i-j,2n-2j).
\eeqn
Furthermore, the component group $A_K(x_j)\coloneqq Z_K(x_j)/Z_K(x_j)^0$ of the centralizer $Z_K(x_j)$ acts trivially on $H^*(\upsilon_i^{-1}(x_j),\bC)$. 
\end{lemma}
\begin{proof}The proof proceeds in the same manner as the proofs of displayed statements \cite[(7.1) and (7.2)]{CVX} for the maps denoted there by $\sigma_i$.

\end{proof}
Making use of Lemma~\ref{eqn-fiber}, one readily checks that for $x_j\in\cO_{2^j1^{2n-2j}}$, 
$
2\dim\,\upsilon_i^{-1}(x_j)-\on{codim}_{\bar\cO_{2^i1^{2n-2i}}}\cO_{2^j1^{2n-2j}}=(i-j)(2n-2i-1)
$. Let us write
\beqn
m_{ij}=(i-j)(2n-2i-1).
\eeqn
It  follows from the decomposition theorem and the lemma above that we have the following decomposition
\beq\label{decomposition sigmai-even}
\upsilon_{i*}\bC[i(2n-i)]\cong\bigoplus_{j=0}^i\bigoplus_{k=0}^{m_{ij}}\IC\left(\bar\cO_{2^j1^{2n-2j}},\bC^{s^i_{jk}}\right)[\pm k]
\eeq
where $s^i_{jk}$ are non-negative integers and $s^i_{ik}=\delta_{0,k}$. The fact that all the local systems appearing in this decomposition are trivial follows from the lemma. 

In what follows we write $(s^i_{jk})_n$ for the numbers $s_{jk}^i$ to indicate that the ambient symmetric pair is $(\rSL(2n),\rSO(2n))$. The lemma below allows us to compute the numbers $(s^i_{jk})_n$ and the stalks of the $\IC(\bar\cO_{2^i1^{2n-2i}},\bC)$ simultaneously by induction.

\begin{lemma}\label{lemma induction}We have
\begin{enumerate}
\item  $(s^i_{lk})_n=(s^{i-j}_{l-j,k})_{n-j}$.
\item  $\mathcal H^{k}_{x_j}\IC(\bar\cO_{2^i1^{2n-2i}},\bC)=\mathcal H^{k+t_j}_0\IC(\bar\cO_{2^{i-j}1^{2n-2i}},\bC)$, where $x_j\in\cO_{2^j1^{2n-2j}}$ and  $t_j=j(2n-j)$.
\end{enumerate}
\end{lemma}
\begin{proof}
In \cite[\S 7.5]{CVX} we proved this lemma in the odd case, i.e., for the symmetric pair $(\rSL(2n+1),\rSO(2n+1))$. This argument can be readily adapted to the even case \linebreak$(\rSL(2n),\rSO(2n))$ we consider in this paper. 
\end{proof}

\subsection{Fourier transforms of \texorpdfstring{$\IC(\bar\cO_{2^i1^{2n-2i}},\bC)$}{lg}.}
Let $\fF:\mathrm{D}_K(\Lg_1)\to \mathrm{D}_K(\Lg_1)$ be the Fourier transform where we have identified $\Lg_1$ with $\Lg_1^*$ via a $K$-invariant non-degenerate bilinear form on $\Lg_1$. Then $\fF$ induces an equivalence of categories $\on{Perv}_K(\Lg_1)\to\on{Perv}_K(\Lg_1)$, where $\on{Perv}_K(\Lg_1)$ is the category of $K$-equivariant perverse sheaves on $\Lg_1$. In this subsection we study the Fourier transforms of $\IC(\bar\cO_{2^i1^{2n-2i}},\bC)$, $i\in[0,n-1]$.

Let us choose a Cartan subspace $\fa$ of $\Lg_1$ such that it consists of diagonal matrices. Let $\fa^{rs}=\fa\cap\Lg^{rs}$ denote the set of regular semisimple elements in $\fa$. For $a\in\fa$ with diagonal entries $a_1,\ldots,a_{2n}$, we write
$a=(a_1,\ldots,a_{2n})$. Thus $a=(a_1,\ldots,a_{2n})\in\fa^{rs}$ if and only if $a_i\neq a_j$ for $i\neq j$.

Consider the family $C\ra\fg_1^{rs}$ whose fiber over $\gamma\in\fg_1^{rs}$ is the 
hyperelliptic curve $C_\gamma$ with affine equation $y^2=\on{det}(t\cdot\on{id}-\gamma)$. The family above 
is constant on $K$-orbits and descends to the 
universal family $\bar C\ra \fg_1^{rs}\sslash K=\fa^{rs}/S_{2n}$ of hyperelliptic curves of genus $n-1$: to each $a=(a_1, \dots,a_{2n})\in \fa^{rs}$ we associate the hyperelliptic curve $\bar C_a$ over $\bP^1$ which ramifies at $\{a_1, \dots,a_{2n}\}$. 
The family 
$C\ra\fg_1^{rs}$
gives us a monodromy representation $$\pi_1^K(\fg_1^{rs},\gamma)\to \pi_1(\fg_1^{rs}\sslash K,\gamma)=B_{2n}\ra \mathrm{Sp}\left(H^1(C_\gamma,\bC)\right).$$ Note that, by \cite{A} (see also \cite{KS}) this monodromy representation has a Zariski dense image. 
From this we get a monodromy representation on the cohomology of the Jacobian of $C_\gamma$ which we break into primitive parts:
\beqn
\pi_1^K(\fg_1^{rs},\gamma) \to \mathrm{Sp}\left(H^i\left(\on{Jac}(C_\gamma),\bC\right)_{\mathrm{prim}}\right)\cong\mathrm{Sp}\left(\left(\wedge^iH^1(C_\gamma,\bC)\right)_{\mathrm{prim}}\right),\ \ i\in[1,n-1]\,.
\eeqn
Associated to this representation $W_i$ we obtain an irreducible $K$-equivarant local system $\calW_i$ on $\fg_1^{rs}$. Note that the part $A[2]$ of $\pi_1^K(\Lg_1^{rs})$ acts trivially on $\calW_i$. It is clear that we have $\calW_i\ncong\calW_j$ for $i\neq j$ and
\beqn
\dim W_i= \binom{2n-2} {i}-\binom {2n-2}{i-2}.
\eeqn

\begin{proposition}\label{prop-matching}
We have
\beqn
\fF\left(\IC(\bar\cO_{2^i1^{2n-2i}},\bC)\right)\cong\IC(\Lg_1,\mathcal{W}_i),\ i\in[0,n-1]
\eeqn
where $\calW_i$ are irreducible $K$-equivariant local systems  on $\Lg_1^{rs}$ defined above and $\calW_0=\bC$ is the trivial local system.
\end{proposition}

\begin{proof}

Consider the map $\upsilon_{n-1}$ defined in \eqref{resolutions}. Let us write $\upsilon=\upsilon_{n-1}$ and 
\beqn
E=\left\{(x,0\subset V_{n-1}\subset V_{n-1}^\p\subset\bC^{2n})\,|\,x\in\Lg_1,\ xV_{n-1}^\p=0\right\}.
\eeqn 
Then   
$
\upsilon:E\to\bar\cO_{2^{n-1}1^2}
$ is the natural projection. 
It follows from \eqref{decomposition sigmai-even} that we have
\beq\label{decomp-sigma}
\upsilon_{*}\bC[n^2-1]\cong\bigoplus_{j=0}^{n-1}\bigoplus_{k=0}^{n-j-1}\IC\left(\bar\cO_{2^j1^{2n-2j}},\bC^{(s^{n-1}_{jk})_n}\right)[\pm k].
\eeq
Let $E^\p$ be the orthogonal complement of $E$ in the trivial bundle $\Lg_1\times X\to X$, where $X=\{V_{n-1}\,|\,0\subset V_{n-1}\subset V_{n-1}^\p\subset\bC^{2n}\}\cong\on{OGr}(n-1,2n)$, i.e., 
\beqn
E^\p=\{(x,0\subset V_{n-1}\subset V_{n-1}^\p\subset\bC^{2n})\,|\,x\in\Lg_1,\,xV_{n-1}\subset V_{n-1}^\p\}.
\eeqn Consider the natural projection map $\check{\upsilon}:E^\p\to\Lg_1$.
By  functoriality of the Fourier transform, we obtain:
\beq\label{FT}
\fF(\upsilon_*\bC[-])\cong\check{\upsilon}_*\bC[-].
\eeq
We show that  
\beq\label{decomp-sigmac}
\begin{gathered}
\check{\upsilon}_*\bC[-]=\left(\bigoplus_{k=0}^{[\frac{n-1}{2}]}\IC\left(\Lg_1,\bigoplus_{j=0}^k\calW_{2j}\right)[\pm(n-1-2k)]\right)
\\
\hspace{1.5in}\oplus\left(\bigoplus_{k=0}^{[\frac{n-1}{2}]-1}\IC\left(\Lg_1,\bigoplus_{j=0}^k\calW_{2j+1}\right)[\pm(n-2-2k)]\right)\oplus\cdots
\end{gathered}
\eeq
where the omitted term $\cdots$ consists of IC complexes that are supported on proper subsets of $\Lg_1$. Using \eqref{FT} and comparing \eqref{decomp-sigma} with \eqref{decomp-sigmac}, we conclude that the proposition holds. It remains to prove~\eqref{decomp-sigmac}.

For any $\gamma\in\fg_1^{rs}$ let
\beqn
F_\gamma=\check{\upsilon}^{-1}(\gamma)\cong\{0\subset V_{n-1}\subset V_{n-1}^\p\subset\bC^{2n}\,|\,\gamma\,V_{n-1}\subset V_{n-1}^\p\},
\eeqn
which is the Fano variety of $(n-2)$-planes in the smooth complete intersection of the two quadrics $Q=0$ and $Q_\gamma=\langle\gamma-,-\rangle_Q=0$. Recall the hyperelliptic curve $C_\gamma$. 
According to \cite{Re,Do,W}, 
there is a canonical action of $\Jac(C_\gamma)$ on $F_\gamma$ such that 
$F_\gamma$ becomes a $\Jac(C_\gamma)$-torsor under this action. This action extends to families: as $\gamma$ varies over $\fg_1^{rs}$, we obtain a $\Jac(C)$-torsor $\check{\upsilon}|_{\check{\upsilon}^{-1}(\fg_1^{rs})}:F\ra\fg_1^{rs}$
of Fano varieties of $(n-2)$-planes in complete intersections of two quadrics. By taking cohomologies of fibers, the families $\Jac(C)$ and $F$ give rise to local systems on $\fg_1^{rs}$. We claim that those local systems coincide, i.e., that 
\beq
\label{mon F}
\begin{gathered}
\text{The $\pi_1^K(\fg_1^{rs},\gamma)$-representations
$H^i(\Jac(C_\gamma),\bC)$ \ \text{and}\  $H^i(F_\gamma,\bC)$}\\\text{are canonically isomorphic.}
\end{gathered}
\eeq
As $F$ does not appear to have a natural section we argue as follows. According to \cite{W} there is a
canonical involution $\sigma$ on $F_\gamma$ compatible with the inversion map on $\Jac(C_\gamma)$.
Let $F_\gamma^\sigma$ be the set of $\sigma$-fixed points on $F_\gamma$. Then $F_\gamma^\sigma$
is naturally a $\Jac(C_\gamma)[2]$-torsor. Thus, the $\Jac(C)$-torsor $F$ gives rise to a 
$\Jac(C)[2]$-torsor $F^{\sigma}\ra\fg_1^{rs}$ consisting of $\sigma$ fixed points on $F$.
We note that there is a canonical isomorphism \[(\Jac(C)\times_{\fg_1^{rs}} F^\sigma)/\Jac(C)[2]\cong F\] where $\Jac(C)[2]$ acts on $\Jac(C)\times_{\fg_1^{rs}} F^\sigma$ via the 
diagonal action. Moreover, arguing as in \cite[Lemma 5.5]{CVX}, we obtain \eqref{mon F}. 
We then conclude that 
\beqn
(R^i\check{\upsilon}_*\bC|_{\Lg_1^{rs}})_{\mathrm{prim}}\cong\calW_i, \ i\in[1,n-1].
\eeqn
Thus~\eqref{decomp-sigmac} follows.
\end{proof}

\section{Cohomology of Fano varieties}
\label{second section}
In this section we compute the cohomology of the Fano varieties $\Fano {k}$ of $k$-planes contained in the smooth complete intersection of two quadrics in $\bP^{2n-1}$, making use of the results in \S\ref{first section}.

\subsection{Fano varieties}Consider the natural projection maps 
\beqn
\check{\upsilon}_i:\{(x,0\subset V_i\subset V_i^\p\subset\bC^{2n})\,|\,x\in\Lg_1,\ xV_i\subset V_i^\p\}\to\Lg_1.
\eeqn
For $\gamma\in\Lg_1^{rs}$, the fiber $\check{\upsilon}_i^{-1}(\gamma)$ can be identified with the Fano variety $\Fano {i-1}$ of $(i-1)$-planes contained in the smooth complete intersection of the two quadrics $Q=0$ and $Q_\gamma=\langle\gamma-,-\rangle_Q=0$ in $\bP^{2n-1}$.  It is easy to see that
\beq\label{dimension}
d_i\coloneqq\dim\Fano {i-1}=\dim\check{\upsilon}_i^{-1}(\gamma)=i(2n-2i-1).
\eeq
Consider $\pi_i=\check{\upsilon}_i|_{\check{\upsilon}_i^{-1}(\Lg_1^{rs})}$, which is a smooth family of Fano varieties, and consider the corresponding local systems $R^k\pi_{i*}\bC$. 
Utilizing functoriality of the Fourier transform, just as in \eqref{FT}, we have:
\beqn 
\fF(\check\upsilon_{i*}\bC[-])\cong{\upsilon}_{i*}\bC[-].
\eeqn
Together with \eqref{decomposition sigmai-even}, this implies that
\beqn
\begin{gathered}
\fF\left(\check\upsilon_{i*}\bC[-]\right)\cong\bigoplus_{k=0}^{2d_i}\fF\left(\IC(\fg_1,R^k\pi_{i*}\bC)[-k+d_i]\right)\\
\cong
\bigoplus_{j=0}^i\bigoplus_{k=0}^{(i-j)(2n-2i-1)} \IC\left(\bar\cO_{2^{j}1^{2n-2j}},\bC^{(s^{i}_{jk})_n}\right)[\pm k].
\end{gathered}
\eeqn
Hence we see that 
$\fF\left(\IC(\fg_1,R^k\pi_{i*}\bC)\right)$ is supported on $\bar\cO_{2^i1^{2n-2i}}$, and 
has the form 
\[\fF\left(\IC(\fg_1,R^k\pi_{i*}\bC)\right)\cong\bigoplus_{j=0}^i \IC\left(\bar\cO_{2^j1^{2n-2j}},\bC^{(s^i_{j,|d_i-k|})_n}\right).\]
It follows from the discussion above and Proposition \ref{prop-matching} that the cohomology of the Fano variety $\Fano {i-1}$
is described as follows
\beq\label{eqn coho of fano even}
H^k(\Fano {i-1},\bC)\cong\bigoplus_{j=0}^i (s^i_{j,|d_i-k|})_n\,W_j.
\eeq
It remains to determine the numbers $(s_{jk}^{i})_n$.

\subsection{The numbers \texorpdfstring{$(s_{jk}^{i})_n$}{lg} and stalks for \texorpdfstring{$\IC({\cO}_{2^i1^{2n-2i}},\bC)$}{lg} at \texorpdfstring{$0$}{lg}}\label{numbers}
In this subsection, we give explicit formulas for the numbers $(s_{jk}^{i})_n$ and the dimensions of the stalks $\mH^k_0\IC({\cO}_{2^i1^{2n-2i}},\bC)$ in terms of their generating functions. Note first that Lemma \ref{lemma induction} (1) implies that 
\beq
\label{eq reduction}
(s^i_{jk})_n=(s^{i-j}_{0,k})_{n-j}\text{ for }j\geq 1.
\eeq
 Since $(s_{i,k}^i)_n=\delta_{0,k}$, it suffices to study the numbers $(s^i_{0,k})_n$. 

For $0\leq j\leq n-1$, let us define   
\beq\label{stalks}
h_j^{(n)}(q):=(-1)^jq^{j(2n-j)}\sum_{k}(-1)^k\left(\dim \mH^{k}_0\,\IC(\bar\cO_{2^{j}1^{2n-2j}},\bC)\right)q^k,
\eeq
and 
\beq\label{fn-numbers}
P_j^{(n)}(q):=(-1)^jq^{j(2n-2j-1)}\sum_{k=0}^{j(2n-2j-1)}(-1)^k(s^j_{0k})_n\,q^{\pm k}.
\eeq
Note that $h_0^{(n)}(q)=P_0^{(n)}(q)=1$ and 
\beq\label{eqn-properties}
\begin{gathered}
\text{For $j\geq 1$, the function $h_j^{(n)}(q)$ (resp. $P_j^{(n)}(q)$) is a polynomial in $q$ with}\\\text{degree no greater than $j(2n-j)-1$ (resp. $2j(2n-2j-1)$)}.
\end{gathered}
\eeq
\beq\label{eqn-properties2}
\text{The coefficients of $q^{j(2n-2j-1)+k}$ and $q^{j(2n-2j-1)-k}$ in $P_j^{(n)}(q)$ are equal for all $k\geq 1$.}
\eeq
The statement above for $h_j^{(n)}(q)$, $j\geq 1$, follows from the fact that $\mH^{k}_0\,\IC(\bar\cO_{2^{j}1^{2n-2j}},\bC)$ is non-zero only if $-\dim\cO_{2^{j}1^{2n-2j-1}}=-j(2n-j)\leq k\leq -1$.

For $0<i\leq k$, let us write $g_{i,k}(q)$ for the Poincare polynomial of the Grassmannian variety of $i$-dimensional subspaces in $\bC^k$, i.e.,
\beq\label{grassm1}
g_{i,k}(q)=\frac{\prod_{s=k-i+1}^{k}(1-q^{2s})}{\prod_{s=1}^i(1-q^{2s})}.
\eeq
Let us also define
\beq\label{grassm2}
g_{0,k}(q)=1\text{ and }g_{i,k}(q)=0\text{ for }i<0.
\eeq
\begin{proposition}
We have
\beq
\label{formula for h}
h_j^{(n)}(q)=  \sum_{k=0}^{[j/2]}\left(\left(\prod_{s=1}^{j-2k}(1+q^{2n-2s})\right)\cdot \left(g_{k,n-j+2k-1}(q^2)-g_{k-1,n-j+2k-1}(q^2)\right)\right),
\eeq
\beq\label{eqn-numbers}
P_j^{(n)}(q)=g_{j,2n-1-j}(q).
\eeq
 
\end{proposition}

\begin{proof}
Note that $\upsilon_i^{-1}(0)\cong\on{OGr}(i,2n)$. Let us write
\beqn
og_{j,2n}(q)=\sum_{k=0}^{r_j}(-1)^k\on{dim} H^{k}\left(\on{OGr}(j,2n),\bC\right)q^k,
\eeqn
where $r_j=2\dim\on{OGr}(j,2n)=j(4n-3j-1)$. The polynomials $og_{j,2n}(q)$ are well-known, i.e., 
\beqn
og_{j,2n}(q)=\frac{(1-q^{2n})\prod_{k=n-j}^{n-1}(1-q^{4k})}{(1-q^{2(n-j)})\prod_{k=1}^j(1-q^{2k})}.
\eeqn

 Taking stalks $\mH^k_0$ on both sides of the equation \eqref{decomposition sigmai-even}, we obtain that
\beq
\label{inductive formula}
\begin{gathered}
(-1)^{i}\,og_{i,2n}(q)\,q^{-i(2n-i)}=\qquad\qquad\qquad \qquad   \\
\sum_{j=0}^{i}\,\left(\sum_{k}(-1)^k\left(\dim \mH^{k}_0\,\IC(\bar\cO_{2^{j}1^{2n-2j}},\bC)\right)q^k\right)\cdot\left(\sum_{k=0}^{(i-j)(2n-2i-1)}\,(-1)^k\,(s^i_{jk})_n\,q^{\pm k}\right).
\end{gathered}
\eeq
Making use of \eqref{eq reduction}, the equation \eqref{inductive formula} can be written in terms of the functions defined in \eqref{stalks} and \eqref{fn-numbers} as 
\beq\label{inductive eqn}
og_{i,2n}(q)=\sum_{j=0}^{i}q^{(i-j)(i-j+1)}\,h_j^{(n)}(q)\,P_{i-j}^{(n-j)}(q).
\eeq
Observe that
\beq\label{observation}
\begin{gathered}
 \text{The equations \eqref{inductive eqn} determine the polynomials }h_i^{(n)}(q) \text{ and }P_i^{(n)}(q)\text{ uniquely}\\
 \text{given that }h_i^{(n)}(q)\text{ and }P_i^{(n)}(q)\text{ satisfy \eqref{eqn-properties} and \eqref{eqn-properties2} }.
 \end{gathered}
\eeq
In fact, this follows from a simple induction argument. Given that $P_i^{(n)}(q)$ satisfy \eqref{eqn-properties2}, it suffices to determine the coefficient of $q^k$ in $P_i^{(n)}(q)$ for $k\geq i(2n-2i-1)$, or equivalently, the coefficient of $q^k$ in $q^{i(i+1)}P_i^{(n)}(q)$ for $k\geq i(2n-i)$. By induction, we can assume that in \eqref{inductive eqn}, $h_j^{(m)}(q)$ and $P_{j}^{(m)}(q)$ are known for all $j<i$ and all $m$. Now the degree of $h_i^{(n)}(q)$ is at most $ i(2n-i)-1$. Thus \eqref{inductive eqn} determines $P_i^{(n)}(q)$ and $h_i^{(n)}(q)$ uniquely.

In view of the observation \eqref{observation}, the proposition follows from Theorem \ref{mainformula} and Corollary \ref{maincor} in Appendix \ref{appendix}. Note that $g_{i,k}(q)=\qbin{k}{i}{q^2}$ in the notation of Definition \ref{defn-a}.

\end{proof}

\begin{corollary}
We have
\begin{enumerate}
\item
 $h_i^{(n)}(q)\in \bZ_{\geq 0}[q^2]$.
\item
$\mH^{k}\IC(\bar\cO_{2^{i}1^{2n-2i}},\bC)=0$ if $k\equiv i-1\ (\on{mod}\ 2)$.
\end{enumerate}
\end{corollary}
\begin{proof}
For (1) it suffices to show that 
$h_i^{(n)}(q)\in\bZ[q^2]$. This  
follows from \eqref{inductive eqn} and induction on $i$. See also Corollary \ref{maincor} for a direct proof. To prove (2) 
we observe that (1) implies $\mH^{k}_0\IC(\bar\cO_{2^{i}1^{2n-2i}},\bC)=0$ if $k\equiv i-1\ (\on{mod}\ 2)$. Now the desired claim follows from 
Lemma \ref{lemma induction}.

\end{proof}
\begin{remark}
As noted in the introduction, it would be interesting to understand the formula \eqref{inductive eqn} in geometric terms. The formula involves doubling certain cohomological degrees and we have been unable to come up with a geometric interpretation for this phenomenon. 
\end{remark}

\begin{remark}
As discussed in the introduction, the functions $h_{m}^{(n)}(q)$ can be viewed as a symmetric space analogue of Kostka--Foulkes polynomials. In \cite[Theorem 7.1, Lemma 8.3]{CVX} we have shown that 
in the odd case, i.e., for the symmetric pair $(\rSL(2n+1),\rSO(2n+1))$, the local IC on nilpotent orbits of order two 
are isomorphic to 
the local IC on nilpotent orbits of order two in $\mathfrak{sp}(2n)$. In turn, we deduced that the corresponding ``Kostka--Foulkes polynomials" are given by the fake degree polynomials, which have a much simpler form.
\end{remark}

\subsection{Proof of Theorem \ref{thm-fano}}

Recall that $d_{i}=\dim\Fano {i-1}=i(2n-2i-1)$ (see \eqref{dimension}). Let us write $\bar\wedge^{j}=\bar\wedge^{j}\left(H^1(C,\bC)\right)=\wedge^{n-1-j}\left(H^1(C,\bC)\right)$.  
It follows from \eqref{eqn coho of fano even} and \eqref{eq reduction} that we have
\beqn
{H}^{d_{n-i}-k}(\Fano{n-i-1},\bC)\cong\bigoplus_{j=0}^{n-i} (s^{n-i}_{j,|k|})_nW_j\cong\bigoplus_{j=0}^{n-i}\left( (s^{n-i}_{j,|k|})_n-(s^{n-i}_{j+2,|k|})_n\right)\wedge^{j}\left(H^1(C,\bC)\right)
\eeqn
\beqn
\cong\bigoplus_{j=0}^{n-i}\left( (s^{n-i-j}_{0,|k|})_{n-j}-(s^{n-i-j-2}_{0,|k|})_{n-j-2}\right)\bar\wedge^{n-1-j}\cong\bigoplus_{j=i-1}^{n-1}\left( (s^{j-i+1}_{0,|k|})_{j+1}-(s^{j-i-1}_{0,|k|})_{j-1}\right)\bar\wedge^{j}.
\eeqn
Here we have used the convention that $(s^a_{0,b})_m=0$ if $a<0$. 
Thus we have  
$$N_i(k,j)=(s^{j-i+1}_{0,|k|})_{j+1}-(s^{j-i-1}_{0,|k|})_{j-1}.$$ 
Using \eqref{fn-numbers} and \eqref{eqn-numbers}, we see that 
\begin{itemize}
\item[] $\text{$(s^{j-i+1}_{0,|k|})_{j+1}$ is the coefficient of $q^k$ (or $q^{-k}$) in 
$(-1)^{j-i+1+k}q^{-(j-i+1)(2i-1)}g_{j-i+1,i+j}(q)$},$\\
\item[]
$\text{$(s^{j-i-1}_{0,|k|})_{j-1}$ is the coefficient of $q^k$ (or $q^{-k}$) in 
$(-1)^{j-i+1+k}q^{-(j-i-1)(2i-1)}g_{j-i-1,i+j-2}(q)$}.$
\end{itemize}
Using \eqref{grassm1} and \eqref{grassm2}, one readily checks that $N_i(k,j)$ is the coefficient of $q^k$ (or $q^{-k}$) in 
\beqn
(-1)^{j-i+1+k}q^{-(j-i+1)(2i-1)}(1-q^{4j})\frac{\prod_{l=j-i+2}^{i+j-2}(1-q^{2l})}{\prod_{l=1}^{2i-2}(1-q^{2l})}.
\eeqn
Note that such a coefficient is nonzero only if $k\equiv j-i+1$ (mod $2$); in the latter case $(-1)^{j-i-1+k}=1$. This proves the theorem as $n=g+1$.

\begin{remark}
For $i=2$, the formula in Theorem \ref{thm-fano} coincides with the formula in \cite[Theorem 1]{N}. 

\end{remark}

\begin{example} 
The cohomology of $\Fano 1$, the Fano variety of lines in the smooth complete intersection of two quadrics in $\bP^{2n-1}$
can be described as follows:
\beqn
H^{8n-20-k}(\Fano 1,\bC)\cong H^{k}(\Fano 1,\bC)
\eeqn
\beqn
\cong\left\{\begin{array}{ll}\bC^{[\frac{m+2}{2}]}&\text{ if } k=2m\text{ and }0\leq m\leq 2n-6,\\H^1(C,\bC)&\text{ if }k=2m+1\text{ and }n-3\leq m\leq 2n-6,\\\bC^{n-3}\oplus\wedge^2\left(H^1(C,\bC)\right)&\text{ if }k=4n-10.\end{array}\right.
\eeqn
\end{example}

\appendix

\section{The functions \texorpdfstring{$h_m^{(n)}(q)$}{lg}}
\label{appendix}
\begin{center}
{\em by Dennis Stanton\footnote{School of Mathematics, University of Minnesota, USA. E-mail: \texttt{stant001@umn.edu}.}}
\end{center}

The main results for $h_m^{(n)}(q)$ are given in 
Theorem~\ref{mainformula} and Corollary~\ref{maincor}. 

\begin{defn}\label{defn-a} 
For non-negative integers $n$ and $k$  let
$$
(A;q)_n=\prod_{k=0}^{n-1}(1-Aq^k), \quad 
\qbin{n}{k}{q}= \frac{(q^n;q^{-1})_k}{(q;q)_k}.
$$
\end{defn}

Note that the $q$-binomial coefficient $\qbin{n}{k}{q}$ is known to be 
a polynomial in $q$ of degree 
$k(n-k)$ with non-negative coefficients, see \cite[Theorem 3.2, p. 35]{And2}.

\begin{defn} For $0\le k\le n-1$ let
$$
og_{k,2n}(q)=\frac{(q^{4(n-k)};q^4)_k}{(q^2;q^2)_k}
\frac{1-q^{2n}}{1-q^{2(n-k)}}.
$$
\end{defn}

\begin{defn} 
\label{defn3}
For $n\ge 1$ and $0\le j\le n-1$ let $h_j^{(n)}(q)$
for $0\le j\le n-1$ be defined recursively by the $n$ equations 
$$
og_{k,2n}(q)=\sum_{j=0}^k q^{(k-j)(k-j+1)}h_j^{(n)}(q) \qbin{2n-1-k-j}{k-j}{q^2},
\quad 0\le k\le n-1.
$$
\end{defn}

An explicit formula for $h_m^{(n)}(q)$ is the main result. 
The proof is given at the end of the Appendix. 

\begin{theorem} 
\label{mainformula}
If $0\le m\le n-1,$ then
$$
\begin{aligned}
h_m^{(n)}(q)= & \sum_{k=0}^{[m/2]} (-q^{2n-2};q^{-2})_{m-2k}
\qbin{n-m+2k}{k}{q^4}
\frac{1-q^{4(n-m)}}{1-q^{4(n-m+2k)}}\,q^{4k}
\\
=&
\sum_{k=0}^{[m/2]} (-q^{2n-2};q^{-2})_{m-2k}
\left( \qbin{n-m+2k-1}{k}{q^4}-\qbin{n-m+2k-1}{k-1}{q^4}\right),
\end{aligned}
$$
where 
$$
\qbin{n-m-1}{-1}{q^4}=0.
$$
\end{theorem}

\begin{rmk} Note that Theorem~\ref{mainformula} implies the recurrence
$$
h_m^{(n)}(q)=
\begin{cases}
(1+q^{2n-2})h_{m-1}^{(n-1)}(q), \quad {\text{if $m\ge 1$ is odd}},\\
(1+q^{2n-2})h_{m-1}^{(n-1)}(q)+\qbin{n-1}{m/2}{q^4}-\qbin{n-1}{m/2-1}{q^4},
\quad {\text{if $m\ge 2$ is even}}.
\end{cases}
$$
\end{rmk}

The polynomiality and positivity of $h_m^{(n)}(q)$ follows 
from Theorem~\ref{mainformula}.

\begin{cor}
\label{maincor}
If $0\le m\le n-1,$ then $h_m^{(n)}(q)$ is a polynomial in $q$ of degree 
$m(2n-m-1)$ with non-negative integer coefficients.
\end{cor}
\begin{proof}
In fact we show that the $k^{th}$ term in the sum of Theorem~\ref{mainformula} 
is a non-negative polynomial in $q$ of degree $m(2n-m-1)-2k.$ 

The factor
$$
(-q^{2n-2};q^{-2})_{m-2k}=\prod_{j=0}^{m-2k-1}(1+q^{2n-2-2j})
$$
is a non-negative polynomial in $q$ of degree $(2n-2)(m-2k)-2\binom{m-2k}{2}.$ 

The factor 
$$
\qbin{n-m+2k-1}{k}{q^4}-\qbin{n-m+2k-1}{k-1}{q^4}
$$
is a non-negative polynomial in $q^4$ of degree $k(n-m+k-1).$ 
This difference is known to be non-negative, see \cite{And}. 
It is also the fake degree polynomial $f^{\lambda}(q^4),$ 
for a partition $\lambda=(n-m-1+k,k)$ with 2 rows. It may be written 
as the generating function for the major index of standard Young tableaux of shape 
$\lambda$, see \cite[Corollary 7.21.5, p. 376]{St}. 

Since 
$$
(2n-2)(m-2k)-2\binom{m-2k}{2}+4k(n-m+k-1)=m(2n-m-1)-2k,
$$
the $k^{th}$ term has degree $m(2n-m-1)-2k$ and is non-negative. 
\end{proof}

The proof of Theorem~\ref{mainformula} is in two steps. First, 
an explicit formula for $h_m^{(n)}(q)$ is found by inverting the matrix 
in Definition~\ref{defn3}, see Proposition~\ref{firstprop}. Next a 
basic hypergeometric transformation, Proposition~\ref{q-Bailey},
is applied to obtain Theorem~\ref{mainformula}.

First we have a matrix inverse result which is \cite[Theorem 3.2]{GS}.

\begin{prop} 
\label{matrixinv} 
Suppose that
$$
\beta_k=\sum_{j=0}^k \frac{\alpha_j}{(Q;Q)_{k-j} (AQ;Q)_{k+j}}, \quad 0\le k\le n-1.
$$
Then for $0\le m\le n-1$,
$$
\alpha_m=\sum_{k=0}^m (-1)^{m+k}\beta_k \frac{(AQ;Q)_{m+k-1}}{(Q;Q)_{m-k}} (1-AQ^{2m}) 
Q^{\binom{m-k}{2}}.
$$
\end{prop}

We apply Proposition~\ref{matrixinv} to Definition~\ref{defn3}.  
Fix $n\ge 1.$ Rewrite Definition~\ref{defn3} as
$$
\frac{(-1)^k og_{k,2n}(q)}{(q^{4n-2};q^{-2})_{2k} }=
\sum_{j=0}^k \frac{(-1)^j h_j^{(n)}(q)}{ (q^{-2};q^{-2})_{k-j} (q^{4n-2};q^{-2})_{k+j}},
\ 0\le k\le n-1.
$$
So if  
$$
Q=q^{-2},  \quad A=q^{4n},
\quad \beta_k=\frac{(-1)^k  og_{k,2n}(q)}{(q^{4n-2};q^{-2})_{2k}}, \quad
\alpha_j =(-1)^j h_j^{(n)}(q)
$$
we can apply Proposition~\ref{matrixinv} to solve for $h_m^{(n)}(q).$

\begin{prop} 
\label{firstprop}
If $0\le m\le n-1$, then
$$
\begin{aligned}
h_m^{(n)}(q)&=  \sum_{k=0}^m \frac{og_{k,2n}(q)}{(q^{4n-2};q^{-2})_{2k}}
\frac{(q^{4n-2};q^{-2})_{m+k-1}}{(q^{-2};q^{-2})_{m-k}} (1-q^{4n-4m})q^{-2\binom{m-k}{2}}
\\
&=
C_m
\sum_{k=0}^m 
\frac{(q^{2m};q^{-2})_k}{(q^{-2};q^{-2})_k}
\frac{(q^{4n-2m};q^{-2})_k}{(q^{4n-2};q^{-4})_k}
\frac{(1-q^{2n})}{(1-q^{2(n-k)})}
q^{-k^2-3k}.
\end{aligned}
$$
where
$$
C_m=\frac{(q^{4n-2};q^{-2})_{m-1}}{(q^2;q^2)_m}(1-q^{4n-4m})q^{2m}(-1)^m.
$$
\end{prop}

For the final step in the proof of Theorem~\ref{mainformula}, 
we will need a transformation, Corollary~\ref{q-Bailey}. The next result is used to prove Corollary~\ref{q-Bailey},
and is a quadratic transformation of a basic hypergeometric series. 
\begin{prop}
\label{q-quadtrans}
As formal power series in $x$,
$$
\sum_{n=0}^\infty \frac{(D^2;q)_n}{(q;q)_n}
\frac{(R/q;q^2)_n}{(R/q;q)_n} x^n=
\sum_{k=0}^\infty 
\frac{(D^2;q)_{2k}}{(q^2;q^2)_k (R;q^2)_k}  
\frac{(D^2xq^{2k};q)_\infty}{(x;q)_\infty} q^{k(2k-2)} R^kx^{2k}.
$$
\end{prop}
\begin{proof} We find the coefficient of $x^n$ on the right side,  
and show that it equals the coefficient of $x^n$ on the left side. 
Use the $q$-binomial theorem \cite[Theorem 2.1, p. 17]{And2} to expand
$$
\frac{(D^2xq^{2k};q)_\infty}{(x;q)_\infty}=
\sum_{j=0}^\infty \frac{(D^2q^{2k};q)_j}{(q;q)_j}x^j.
$$ 
So the coefficient of $x^n$ on the right side is
$$
\begin{aligned}
\sum_{k=0}^{[n/2]} &
\frac{(D^2;q)_{2k}}{(q^2;q^2)_k (R;q^2)_k}  
q^{k(2k-2)} R^k\frac{(D^2q^{2k};q)_{n-2k}}{(q;q)_{n-2k}}\\
&= \frac{(D^2;q)_n}{(q;q)_n}
\sum_{k=0}^{[n/2]} \frac{(q^n;q^{-1})_{2k}}{(q^2;q^2)_k (R;q^2)_k}q^{k(2k-2)} R^k\\
&= \frac{(D^2;q)_n}{(q;q)_n}
\sum_{k=0}^{[n/2]} \frac{(q^n;q^{-2})_{k}(q^{n-1};q^{-2})_{k}}{(q^{-2};q^{-2})_k (1/R;q^{-2})_k}q^{-2k}\\
&=  \frac{(D^2;q)_n}{(q;q)_n}
\frac{(R/q;q^2)_n}{(R/q;q)_n},
\end{aligned}
$$
where we have used the $q$-Vandermonde theorem 
\cite[(II.6), p. 354]{GR}, to evaluate the last sum.
\end{proof}

The transformation we need is a corollary of 
Proposition~\ref{q-quadtrans}, and is a $q$-analogue of 
a result of Bailey~\cite[(5.41)]{Bai}.

\begin{cor}
\label{q-Bailey} 
If $m$ is a non-negative integer, then
$$
\begin{aligned}
\sum_{k=0}^{[m/2]} &
\frac{(D^2;q)_{2k}}{(q^2;q^2)_k (R;q^2)_k}  
q^{k(2k-2)} R^k 
\frac{(D^2q^{2k}/B;q)_{m-2k}}{(q;q)_{m-2k}}B^{m-2k}\\
=&
\sum_{s=0}^m \frac{(D^2;q)_{m-s}}{(q;q)_{m-s}}
\frac{(R/q;q^2)_{m-s}}{(R/q;q)_{m-s}}
\frac{(1/B;q)_{s}}{(q;q)_{s}}B^{s}\\
=&
\frac{(D^2;q)_{m}}{(q;q)_{m}}
\frac{(R/q;q^2)_{m}}{(R/q;q)_{m}}
\sum_{s=0}^m \frac{(q^{-m};q)_s}{(q^{1-m}/D^2;q)_s}
\frac{(q^{2-m}/R;q)_s}{(q^{3-2m}/R;q^2)_s}
\frac{(1/B;q)_{s}}{(q;q)_{s}}\left(\frac{Bq^{2-m}}{D^2}\right)^{s} q^{\binom{s}{2}}
\end{aligned}
$$
\end{cor}
 
\begin{proof} Multiply both sides of Proposition~\ref{q-quadtrans}
by $(x;q)_\infty/(Bx;q)_\infty$ and equate coefficients of $x^m$, using
$$
\frac{(x;q)_\infty}{(Bx;q)_\infty}=\sum_{j=0}^\infty \frac{(1/B;q)_j}{(q;q)_j}(Bx)^j,
\quad
\frac{(D^2xq^{2k};q)_\infty}{(Bx;q)_\infty}=
\sum_{j=0}^\infty \frac{(D^2q^{2k}/B;q)_j}{(q;q)_j}(Bx)^j.
$$
The last equality uses 
$$
(A;Q)_{m-s}=\frac{(A;Q)_m}{(Q^{1-m}/A;Q)_s \left(-AQ^{m-1}\right)^s Q^{-\binom{s}{2}}}.
$$
\end{proof}

We use Corollary~\ref{q-Bailey} to give another sum for $h_m^{(n)}(q).$

\begin{theorem} 
\label{firstpositive}
If $0\le m\le n-1,$
$$
h_m^{(n)}(q)=(-q^{2n-2};q^{-2})_m\sum_{k=0}^{[m/2]}
\frac{(q^{2n-2m};q^2)_{2k}}{(q^{4n+4-4m};q^4)_k (q^4;q^4)_k} q^{4k}
$$
\end{theorem}

\begin{proof}In 
Corollary~\ref{q-Bailey} we replace $q$ by $q^{-2}$, and let
$$
B=q^{-2n}, \quad R=q^{-4n+4m-4}, \quad D^2=q^{-2n+2m}.
$$
The final expression in 
Corollary~\ref{q-Bailey} is $$\frac{(q^{-2n+2m};q^{-2})_{m}}{(q^{-2};q^{-2})_{m}}
\frac{(q^{-4n+4m-2};q^{-4})_{m}}{(q^{-4n+4m-2};q^{-2})_{m}}\frac{h_m^{(n)}(q)}{C_m},$$ where $h_m^{(n)}(q)$ and $C_m$ are as in Proposition~\ref{firstprop}.  The left side of Corollary~\ref{q-Bailey} is $$(-1)^mq^{m^2-2nm+m}\sum_{k=0}^{[m/2]}
\frac{(q^{2n-2m};q^2)_{2k}}{(q^{4n+4-4m};q^4)_k (q^4;q^4)_k} q^{4k}.$$ 
One verifies that $$(-1)^mC_m\frac{q^{m^2-2nm+m}(q^{-2};q^{-2})_{m}}{(q^{-2n+2m};q^{-2})_{m}}
\frac{(q^{-4n+4m-2};q^{-2})_{m}}{(q^{-4n+4m-2};q^{-4})_{m}}=(-q^{2n-2};q^{-2})_m.$$
\end{proof}

\begin{proof}[Proof of Theorem~\ref{mainformula}] We rewrite 
the $k^{th}$ summand in Theorem~\ref{firstpositive}
$$
\begin{aligned}
(-q^{2n-2};q^{-2})_m & \frac{(q^{2n-2m};q^2)_{2k}}
{(q^{4n+4-4m};q^4)_k (q^4;q^4)_k} q^{4k}\\
&=
(-q^{2n-2};q^{-2})_{m-2k}
\frac{(-q^{2n-2m};q^2)_{2k}(q^{2n-2m};q^2)_{2k}}
{(q^{4n+4-4m};q^4)_k (q^4;q^4)_k} q^{4k} 
\\
&=
(-q^{2n-2};q^{-2})_{m-2k}
\frac{(q^{4n-4m};q^4)_{2k}}
{(q^{4n+4-4m};q^4)_k (q^4;q^4)_k} q^{4k} 
\\
&= (-q^{2n-2};q^{-2})_{m-2k}
\qbin{n-m+2k}{k}{q^4}
\frac{1-q^{4(n-m)}}{1-q^{4(n-m+2k)}}q^{4k}\\
&=
(-q^{2n-2};q^{-2})_{m-2k}
\left( \qbin{n-m+2k-1}{k}{q^4}-\qbin{n-m+2k-1}{k-1}{q^4}\right).
\end{aligned}
$$

\end{proof}

\end{document}